\documentclass[12pt]{amsart}
\usepackage[utf8]{inputenc}

\usepackage{amsmath,amsfonts,amssymb,amsthm,epsfig,epstopdf,hyperref}
\usepackage[lmargin=1.5in,rmargin=1.5in,tmargin=1.5in,bmargin=1.5in]{geometry}
\usepackage{graphicx}
\usepackage{overpic}
\usepackage{tikz}
\usetikzlibrary{matrix,arrows,decorations.pathmorphing}
\usepackage[version=4]{mhchem}
\usepackage{siunitx}
\usepackage{longtable,tabularx}
\usepackage{todonotes}
\usepackage{stmaryrd}
\usepackage{mathtools}

\newtheorem{thm}{Theorem}[section]
\newtheorem{rem}[thm]{Remark}
\newtheorem{cor}[thm]{Corollary}
\newtheorem{lem}[thm]{Lemma}
\newtheorem{conj}[thm]{Conjecture}

\begin{document}

\title[Lee Spectral Sequence and Unknotting Number]{The Lee Spectral Sequence, Unknotting Number, and the Knight Move Conjecture}
\author{Akram Alishahi}
\thanks{AA was supported by NSF Grant DMS-1505798}
\address{Department of Mathematics, Columbia University, New York, NY 10027}
\email{\href{mailto:alishahi@math.columbia.edu}{alishahi@math.columbia.edu}}

\author{Nathan Dowlin}
\address{Department of Mathematics, Columbia University, New York, NY 10027}
\email{\href{mailto:ndowlin@math.columbia.edu}{ndowlin@math.columbia.edu}}

\date{\today}

\begin{abstract}
We show that the page at which the Lee spectral sequence collapses gives a bound on the unknotting number, $u(K)$. In particular, for knots with $u(K) \le 2$, we show that the Lee spectral sequence must collapse at the $E_{2}$ page. An immediate corollary is that the Knight Move Conjecture is true when $u(K) \le 2$.
\end{abstract}
\maketitle

\section{Introduction}

In \cite{Kh-Intro}, Khovanov defined a bigraded knot invariant $H_{Kh}(K)$ which categorifies the Jones polynomial. This invariant comes in the form of a homology theory based on a planar diagram for a knot and the Frobenius algebra $\mathbb{Q}[X]/X^{2}=0$.

There is a basic structural theory about Khovanov homology known as the Knight Move Conjecture, which can be stated as follows:

\begin{conj}[Knight Move Conjecture, \cite{Kh-Intro}, \cite{BarN2}]

The Khovanov homology of any knot $K$ decomposes as a single `pawn move' pair 

\[ \mathbb{Q}\{ 0, n-1 \} \oplus \mathbb{Q}\{ 0, n+1 \} \]

\noindent
together with a set of knight move pairs 

\[ \bigoplus_{i} \mathbb{Q}\{ l_{i}, m_{i} \} \oplus \mathbb{Q}\{ l_{i}+1, m_{i}+4 \} \]

\end{conj}

\noindent
where $\mathbb{Q}\{i,j\}$ denotes a generator in bigrading $(i,j)$.

In \cite{LeeSS}, Lee defined a deformation of Khovanov homology by changing the Frobenius algebra to $\mathbb{Q}[X]/(X^{2}=1)$. The corresponding complex can be viewed as the original complex with additional differentials, resulting in a spectral sequence $E_{n}(K)$ from Khovanov homology to Lee homology. Lee showed that this spectral sequence is an invariant of the knot $K$, and that it converges to $\mathbb{Q} \oplus \mathbb{Q}$. Rasmussen~\cite{Ras-sg} added to this result that $E_{\infty}(K) = \mathbb{Q}\{ 0, s-1 \} \oplus \mathbb{Q}\{ 0, s+1 \} $, where $s$ is Rasmussen's slice invariant.

The differential $d_{n}$ on $E_{n}(K)$ has bigrading $(1, 4n)$, so the Knight Move Conjecture is true whenever the Lee spectral sequence collapses at the $E_{2}$ page.

In this paper, we construct a lower bound for the unknotting number $u(K)$ using Lee's homology theory, and we apply this bound to prove that the Lee spectral sequence must collapse at the $E_{2}$ page whenever $u(K) \le 2$. 

We technically use a lift of Lee's complex obtained by setting each $X^{2}=t$ as described in \cite{Kh-Frobext}. The resulting homology $H_{Lee}(K)$ is a module over $\mathbb{Q}[X, t]/(X^{2}=t)$, and it consists of two towers $\mathbb{Q}[t] \oplus \mathbb{Q}[t]$ and an $X$-torsion summand $T_{X}(H_{Lee}(K))$. Note that since $X^{2}=t$, $X$-torsion and $t$-torsion are the same, i.e. $T_{X}(H_{Lee}(K))=T_{t}(H_{Lee}(K))$. We define $\mathfrak{u}_{X}(K)$ to be the maximal order of $X$-torsion in $H_{Lee}(K)$.

\begin{thm}

For any knot $K$, $\mathfrak{u}_{X}(K)$ gives a lower bound for the unknotting number of $K$.

\end{thm}

We prove this by defining crossing change maps $f$ and $g$ as shown below such that on homology, both $f_{*} \circ g_{*}$ and $g_{*} \circ f_{*} $ are either equal to $2X$ or $-2X$. The diagrams $D_{+}$ and $D_{-}$ differ at a single crossing $c$, where $D_{+}$ has a positive crossing and $D_{-}$ has a negative crossing.

\begin{figure}[h!]
\centering
\begin{tikzpicture}
\matrix(m)[matrix of math nodes,
row sep=10em, column sep=12em,
text height=1.5ex, text depth=0.25ex]
{C_{Lee}(D_{+})&C_{Lee}(D_{-})\\};
\path[-{stealth}]
(m-1-1) edge[bend left = 15] node[above] {$f$} (m-1-2)
(m-1-2) edge[bend left = 15] node[below] {$g$} (m-1-1);
\end{tikzpicture}
\label{h}
\end{figure}

Using similar chain maps, the first author gives a lower bound for the unknotting number from Bar-Natan homology \cite{A-BN}.

Since $X^{2}=t$, we have

\[  \left \lceil{\mathfrak{u}_{X}(K)/2 }\right \rceil = \mathfrak{u}_{t}(K)\]

\noindent
where $\left \lceil{x }\right \rceil$ is the ceiling of $x$. The variable $t$ keeps track of the Lee filtration, so if we add $1$ to the maximal order of $t$-torsion in $H_{Lee}(K)$, the result is exactly the page at which the Lee spectral sequence collapses.

\begin{thm}

If $K$ is a knot with $u(K) \le 2$ and $K$ is not the unknot, then the Lee spectral sequence for $K$ collapses at the $E_{2}$ page.

\end{thm}

\begin{cor}

The Knight Move Conjecture is true for all knots $K$ with $u(K) \le 2$.

\end{cor}


\section{Background}

In this section we will describe the Khovanov chain complex and the Lee deformation. We will use a notation that makes the module structure clear. 

\subsection{The Standard Khovanov Complex} Assume $L$ be a link in $S^{3}$ with diagram $D \subset \mathbb{R}^{2}$. Let $\mathfrak{C}=\{c_{1}, c_{2}, ..., c_{n}\}$ denote the crossings in $D$, and viewing  $D$ as a 4-valent graph, let $E=\{e_{1}, e_{2}, ..., e_{m}\}$ denote the edges of $D$. The \emph{edge ring} is defined to be 

\[  R := \mathbb{Q}[X_{1}, X_{2},...,X_{m}]/\{X_{1}^{2}=X_{2}^{2}=...=X_{m}^{2}=0 \}   \]

\noindent
with each variable $X_{i}$ corresponding to the edge $e_{i}$. 

Each crossing $c_{i}$ can be resolved in two ways, the 0-resolution and the 1-resolution (see Figure \ref{resolutions}). For each $v \in \{0,1\}^{n}$, let $D_{v}$ denote the diagram obtained by replacing the crossing $c_{i}$ with the $v_{i}$-resolution. The diagram $D_{v}$ is a disjoint union of circles - denote the number of circles by $k_{v}$. The vector $v$ determines an equivalence relation on $E$, where $e_{p} \sim_{v} e_{q}$ if $e_{p}$ and $e_{q}$ lie on the same component of $D_{v}$.

\begin{figure}[ht]
\vspace{4mm}
\centering
\begin{overpic}[width = .8\textwidth]{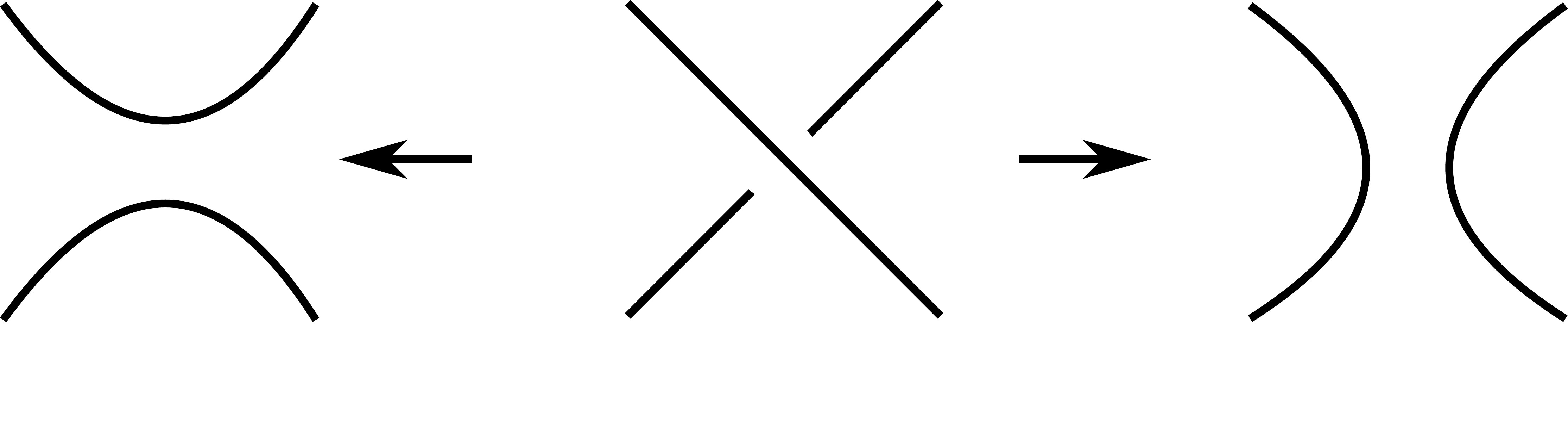}
\put(9.5, 19){$\bullet$}
\put(9.5, 13.75){$\bullet$}
\put(86.3,16.5){$\bullet$}
\put(91.7,16.5){$\bullet$}
\put(0,29){$e_{i}$}
\put(0,4){$e_{j}$}
\put(19, 29){$e_{k}$}
\put(19,4){$e_{l}$}
\put(0,-2){0-resolution}
\put(40,29){$e_i$}
\put(40,4){$e_j$}
\put(59,29){$e_k$}
\put(59,4){$e_l$}
\put(80,29){$e_i$}
\put(80,4){$e_j$}
\put(99,29){$e_k$}
\put(99,4){$e_l$}
\put(80,-2){1-resolution}
\end{overpic}
\caption{}\label{resolutions}
\end{figure}

The module $C_{Kh}(D_{v})$ is defined to be a quotient of the ground ring:

\[ C_{Kh}(D_{v}) := R / \{X_{p}=X_{q} \text{ if } e_{p} \sim_{v} e_{q} \}  \]

\noindent
we will denote this quotient by $R_{v}$.

There is a partial ordering on $\{0,1\}^{n}$ obtained by setting $u \le v$ if $u_{i} \le v_{i}$ for all $i$. We will write  $u \lessdot v$ if  $u \le v$ and they differ at a single crossing, i.e. there is some $i$ for which $u_{i}=0$ and $v_{i}=1$, and $u_{j}=v_{j}$ for all $j \ne i$. Corresponding to each edge of the cube, i.e. a pair $(u \lessdot v)$, there is an embedded cobordism in $\mathbb{R}^{2} \times [0,1]$ from $D_{u}$ to $D_{v}$ constructed by attaching a 1-handle near the crossing $c_{i}$ where $u_{i}<v_{i}$. This cobordism is always a pair of pants, either going from one circle to two circles (when $k_{u}=k_{v}-1$) or from two circles to one circle (when $k_{u}=k_{v}+1$). We call the former a \emph{merge} cobordism and the latter a \emph{split} cobordism.

For each vertex $v$ of the cube, the quotient ring $R_{v}$ is naturally isomorphic to $\mathcal{A}^{\otimes k_{v}}$, where $\mathcal{A}$ is the Frobenius algebra $\mathbb{Q}[X]/(X^{2}=0)$. Recall that the multiplication and comultiplication maps of $\mathcal{A}$ are given as:

\begin{displaymath}
m:\mathcal{A}\otimes_{\mathbb{Q}}\mathcal{A}\rightarrow\mathcal{A}:\begin{cases}
\begin{array}{lcc}
1 \mapsto 1&,&X_{1} \mapsto X\\
X_{1}X_{2} \mapsto 0&,&X_{2} \mapsto X
\end{array}
\end{cases}
\end{displaymath}
and 
\begin{displaymath}
\Delta:\mathcal{A}\rightarrow\mathcal{A}\otimes_{\mathbb{Q}}\mathcal{A}:\begin{cases}
1 \mapsto X_{1}+X_{2}\\
X \mapsto X_{1}X_{2}
\end{cases}
\end{displaymath}




The chain complex $C_{Kh}(D)$ is defined to be the direct sum of the $C_{Kh}(D_{v})$ over all vertices in the cube:

\[ C_{Kh}(D) := \bigoplus_{v \in \{0,1\}^{n}} C_{Kh}(D_{v}) \]

The differential decomposes over the edges of the cube. When $u \lessdot v$ corresponds to a merge cobordism, define 

\[   \delta_{u,v}: C_{Kh}(D_{u}) \to C_{Kh}(D_{v})    \]

\noindent
to be the Frobenius multiplication map, and when $u \lessdot v$ corresponds to a split cobordism, define $\delta_{u,v}$ to be the comultiplication map. In terms of the quotient rings $R_{u}$ and $R_{v}$, the map $m$ is projection, while $\Delta$ is multiplication by $X_{j}+X_{k}$, where $e_{i}$, $e_{j}$, $e_{k}$, $e_{l}$ are the edges at the corresponding crossing as in Figure \ref{resolutions}. Note that $X_{j}+X_{k}=X_{i}+X_{l}$.

If $D_{u}$ and $D_{v}$ differ at crossing $c_{i}$, define $\epsilon_{u,v} = \sum_{j < i } u_{j} $. Then 

\[   \delta = \sum_{u \lessdot v} (-1)^{\epsilon_{u,v}} \delta_{u,v}   \]

The Khovanov complex is bigraded, with a homological grading and a quantum grading. Up to an overall grading shift, the homological grading is just the height in the cube. Setting $|v|= \sum_{i} v_{i}$, $n_{+}$ the number of positive crossings in $D$, an $n_{-}$ the number of negative crossings in $D$, we have 

\[ \mathrm{gr}_{h}(R_{v}) = |v|-n_{-}\]
\noindent
For each vertex $v$ of the cube, the quantum grading of $1\in R_{v}$ is given by 

\[ \mathrm{gr}_{q}(1 \in R_{v}) = n_{+}-2n_{-}+|v|+k_{v} \]

\noindent
and each variable $X_{i}$ has quantum grading $-2$. With respect to the bigrading $(\mathrm{gr}_{h}, \mathrm{gr}_{q})$, the differential $\delta$ has bigrading $(1,0)$. The Khovanov homology $H_{Kh}(D)$ is the homology of this complex

\[ H_{Kh}(D) = H_{*}(C_{Kh}(D), \delta)       \]

\subsection{The Lee Deformation} The Lee deformation on Khovanov homology comes from a small modification of the ring $R$. If we replace $R$ with the ring

\[  R' := \mathbb{Q}[X_{1}, X_{2},...,X_{k}, t]/\{X_{1}^{2}=X_{2}^{2}=...=X_{l}^{2}=t \}   \]

\noindent
and define everything as in the previous section, the result is a complex $C_{Lee}(D)$. The variable $t$ has homological grading $0$ and quantum grading $-4$, so the complex is still bigraded. Note that $R'/(t=0) \cong R$, so $C_{Lee}(D)/(t=0) \cong C_{Kh}(D)$.

The edge maps are still given by projection for $m$ and multiplication by $X_{j}+X_{k}$ for $\Delta$. More precisely, for  $u\lessdot v$ the edge homomorphism $\delta_{uv}:R_u\rightarrow R_v$ is given by
\begin{equation*}
\begin{aligned}[c]
&1 \xmapsto{m} 1\\
&X_{i} \xmapsto{m} X_{i}\\
&X_{j} \xmapsto{m} X_{i}\\
&X_{i}X_{j} \xmapsto{m} t
\end{aligned}
\qquad \hspace{10mm}\text{or}\hspace{10mm}\qquad
\begin{aligned}[c]
&1 \xmapsto{\Delta} X_{j}+X_{k}\\
&X_{i} \xmapsto{\Delta} X_{j}X_{k}+t
\end{aligned}.
\end{equation*}

As before, $e_i$, $e_j$, $e_k$ and $e_l$ are the edges at the corresponding crossing as in Figure \ref{resolutions}. By a minor abuse of notation, we refer to this differential as $\delta$ as well.

\noindent
The Lee homology is defined to be the homology of this complex,

\[   H_{Lee}(D) = H_{*}(C_{Lee}(D), \delta).     \]

\begin{rem}
The actual complex defined by Lee in \cite{LeeSS} is given by $C_{Lee}(D)/(t=1)$. Setting $t=1$ replaces the $q$-grading with a filtration, which induces the Lee spectral sequence. The number of page at which the Lee spectral sequence collapses is $1$ more than the maximal degree of $t$-torsion in $H_{Lee}(D)$. 
\end{rem}

\begin{thm}[\cite{LeeSS}]
If $D$ is a diagram for a knot $K$, then ignoring gradings, $H_{*}(C_{Lee}(D))$ decomposes as 

\[ H_{Lee}(D)) \cong \mathbb{Q}[t] \oplus \mathbb{Q}[t] \oplus T(H_{Lee}(D))   \]

\noindent
where $T(C)$ is the $t$-torsion part of $C$.
\end{thm}

\noindent
In other words, the free part of $H_{*}(C_{Lee}(D))$ is isomorphic to the Lee homology of the unknot, and 

\[ H_{*}(C_{Lee}(D)/(t=1)) \cong \mathbb{Q} \oplus \mathbb{Q} \]

\section{The Crossing Change Map}

Let $D_{+}$ and $D_{-}$ be two diagrams that differ at a single crossing $c$, with $D_{+}$ having a positive crossing and $D_{-}$ having a negative crossing. Let $e_{i}, e_{j}, e_{k}, e_{l}$ be the adjacent edges, as in Figure \ref{CC}. In this section we will define chain maps

\[f: C_{Lee}(D_{+}) \to C_{Lee}(D_{-}) \]

\[g: C_{Lee}(D_{-}) \to C_{Lee}(D_{+}) \]

\noindent
such that for any $1\le\mathsf{i}\le m$, both $f \circ g$ and $g \circ f$ are chain homotopy equivalent to multiplication by $2 X_{\mathsf{i}}$ or $-2X_{\mathsf{i}}$.

\begin{figure}[ht]
\centering
\def\svgwidth{10cm}
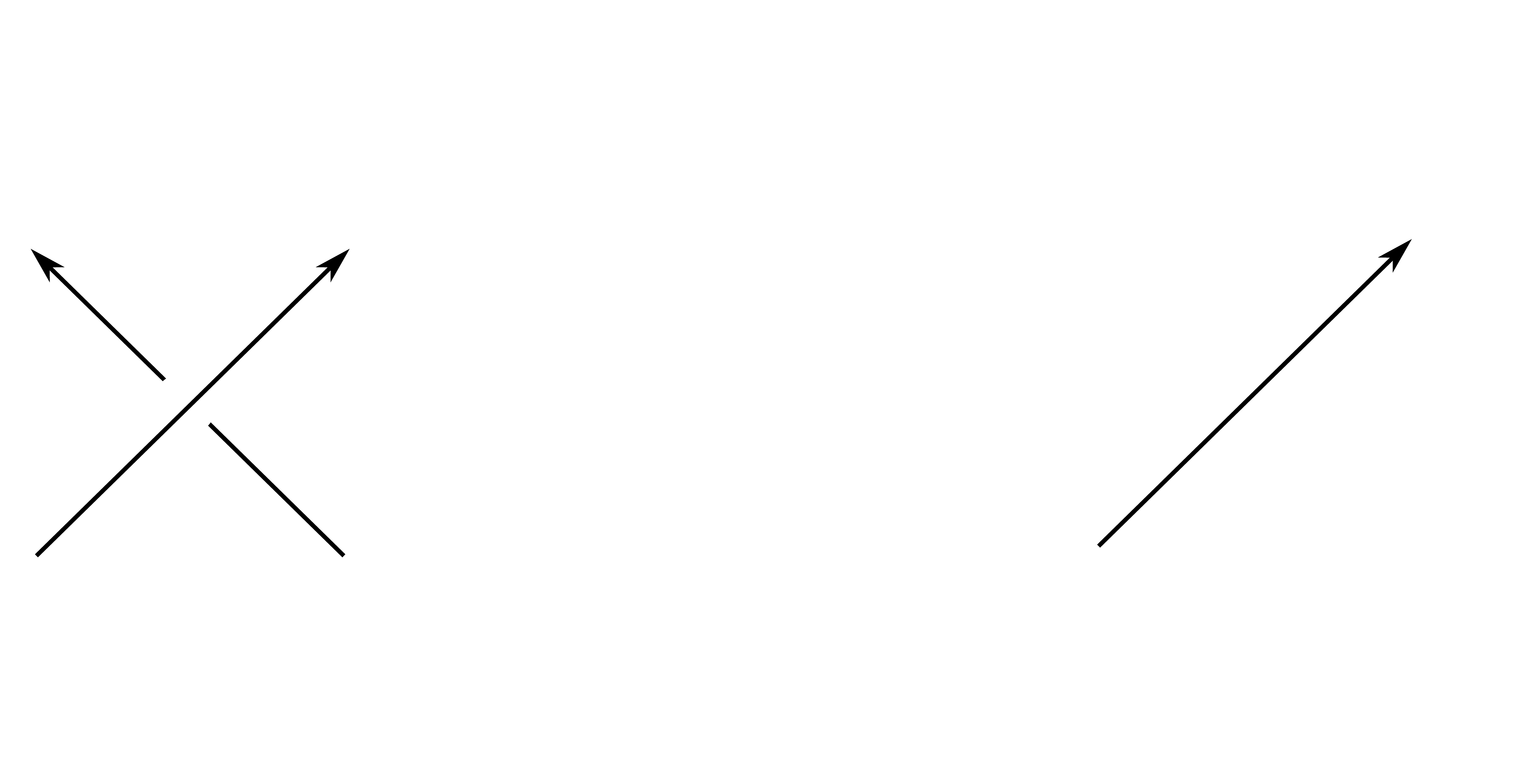
\caption{}\label{CC}
\end{figure}

Let $D^{0}_{+}$ be the $0$-resolution of $D_{+}$ at $c$ and $D^{1}_{+}$ the $1$-resolution $D_{+}$ at $c$, and define $D^{0}_{-}$, $D^{1}_{-}$ analogously. Note that $D^{0}_{+}$ and $D^{1}_{-}$ are the same diagram, as are $D^{1}_{+}$ and $D^{0}_{-}$. We can write 

\[ C_{Lee}(D_{+}) = C_{Lee}(D^{0}_{+}) \xrightarrow{\hspace{3mm} \delta_{+} \hspace{3mm}} C_{Lee}(D^{1}_{+}) \]

\[ C_{Lee}(D_{-}) = C_{Lee}(D^{0}_{-}) \xrightarrow{\hspace{3mm} \delta_{-} \hspace{3mm}} C_{Lee}(D^{1}_{-}) \]

\noindent
where $\delta_{+}$, $\delta_{-}$ are the edge maps for the respective complexes corresponding to the crossing $c$. As modules (ignoring the differentials), we have $C_{Lee}(D_{+}) = D^{0}_{+} \oplus D^{1}_{+}$ and $C_{Lee}(D_{-}) = D^{0}_{-} \oplus D^{1}_{-}$.

For $a$ in $C_{Lee}(D_{+})$, write $a = (a^{0}, a^{1})$. In order to pin down the signs on $C_{Lee}$, we need to choose an ordering of the crossings. For simplicity, take $c$ to be the last crossing. We define $f: C_{Lee}(D_{+}) \to C_{Lee}(D_{-})$ by 

\[  f(a^{0},a^{1}) = ((X_{j}-X_{k})a^{1}, a^{0})    \]

\noindent
Similarly, we define $g: C_{Lee}(D_{-}) \to C_{Lee}(D_{+})$ by 

\[ g(b^{0},b^{1}) = ((X_{j}-X_{k})b^{1}, b^{0})     \]

\noindent
Diagrammatically, these maps are depicted in Figures \ref{f} and \ref{g}. 

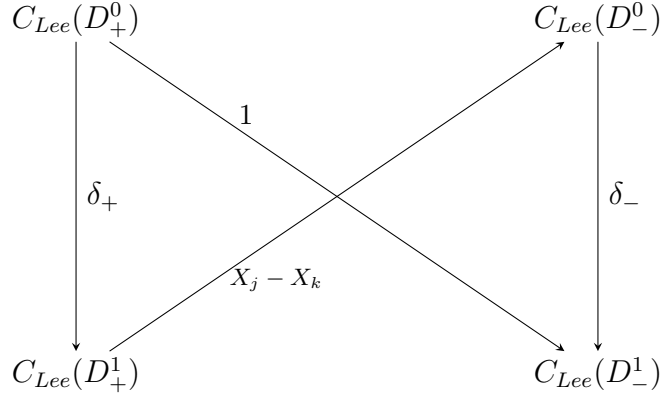
\begin{figure}
\centering
\begin{tikzpicture}
\matrix(m)[matrix of math nodes,
row sep=10em, column sep=12em,
text height=1.5ex, text depth=0.25ex]
{C_{Lee}(D^{0}_{+})&C_{Lee}(D^{0}_{-})\\
C_{Lee}(D^{1}_{+})&C_{Lee}(D^{1}_{-})\\};
\path[-{stealth}]
(m-1-1) edge node[above, pos = .3 ] {\small $1$} (m-2-2)
edge node[auto] {$\delta_{+}$} (m-2-1)
(m-1-2) edge node[auto] {$\delta_{-}$} (m-2-2)
(m-2-1) edge node[below, pos = .3] {\scriptsize \hspace{8mm}$X_{j}-X_{k}$} (m-1-2);
\end{tikzpicture}
\caption{The Chain Map $f$.}
\label{f}
\end{figure}

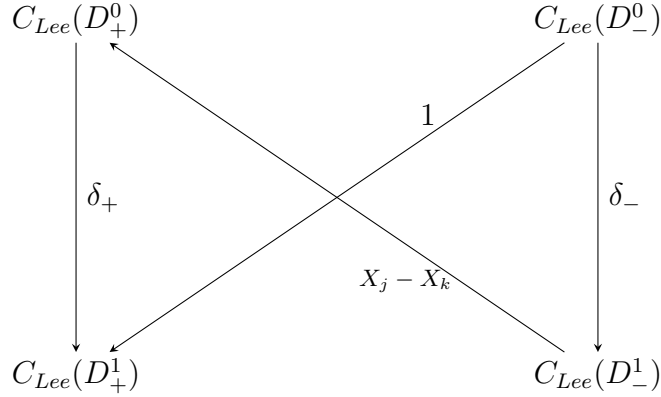
\begin{figure}
\centering
\begin{tikzpicture}
\matrix(m)[matrix of math nodes,
row sep=10em, column sep=12em,
text height=1.5ex, text depth=0.25ex]
{C_{Lee}(D^{0}_{+})&C_{Lee}(D^{0}_{-})\\
C_{Lee}(D^{1}_{+})&C_{Lee}(D^{1}_{-})\\};
\path[-{stealth}]
(m-1-1) edge node[auto] {$\delta_{+}$} (m-2-1)
(m-1-2) edge node[auto] {$\delta_{-}$} (m-2-2)
(m-1-2) edge node[above, pos = .3] {$1$} (m-2-1)
(m-2-2) edge node[below, pos = .3] {\scriptsize \hspace{-7mm} $X_{j}-X_{k}$} (m-1-1);
\end{tikzpicture}
\caption{The Chain Map $g$.}
\label{g}
\end{figure}

\begin{lem}
The maps $f$ and $g$ are chain maps.
\end{lem}

\begin{proof}
We will prove it for $f$ ( The proof for $g$ is completely analogous). The map by $1$ doesn't interact with the edge maps $\delta_{+}$ or $\delta_{-}$, and since we chose $c$ to be the last crossing, the negative signs are the same in $C_{Lee}(D^{0}_{+})$ and in $C_{Lee}(D^{1}_{-})$ in a manner to make the map commute with the respective differentials. Thus, it is a chain map.

Similarly, the negative signs in $C_{Lee}(D^{1}_{+})$ and in $C_{Lee}(D^{0}_{-})$ make the multiplication by $X_{j}-X_{k}$ anticommute with all differentials within these two complexes. Thus, to see that the multiplication by $X_{j}-X_{k}$ is a chain map, we must show $(X_{j}-X_{k})\delta_{+}= \delta_{-}(X_{j}-X_{k})=0$.

The map $\delta_{+}$ is either a merge or a split, depending on the vertex in $C_{Lee}(D^{0}_{+})$. For the vertices where it is a merge, $e_{j}$ and $e_{k}$ will lie on the same circle in the corresponding vertex in $C_{Lee}(D^{1}_{+})$, so $X_{j}=X_{k}$ at these vertices and $(X_{j}-X_{k})\delta_{+}=0$. For the vertices where it is a split, $\delta_{+}$ is multiplication by $X_{j}+X_{k}$, so $(X_{j}-X_{k})\delta_{+} = X_{j}^{2}-X_{k}^{2}= t-t=0$.

The argument that $\delta_{-}(X_{j}-X_{k})=0$ is similar. If we choose a vertex where $\delta_{-}$ is a merge map, then we will have $X_{j}=X_{k}$ on $C_{Lee}(D^{1}_{-})$. But if $\delta_{-}$ is a split map, then $X_{j}-X_{k}$ is already equal to zero on $C_{Lee}(D^{0}_{-})$.

\end{proof}

\begin{lem}

For any $a$ in $C_{Lee}(D_{+})$ and any $b$ in $C_{Lee}(D_{-})$, $(g \circ f)(a)=(X_{j}-X_{k})a$ and $(f \circ g)(b)=(X_{j}-X_{k})b$.

\end{lem}

This is clear from the definitions of $f$ and $g$. The lemma becomes interesting with the following result of Hedden and Ni.

\begin{lem}[\cite{HN-Khovanov}] Let $D$ be a planar diagram for a link $L$. If edges $e_{j}$ and $e_{k}$ are diagonal from one another at a crossing $c$ (positive or negative), then there is a chain homotopy $H: C_{Lee}(D) \to C_{Lee}(D)$ satisfying 

\[ \delta H + H \delta = X_{j}+X_{k}   \]

\end{lem}

Although their proof is for the Khovanov complex, the same argument applies for the Lee complex. For completeness, we will repeat it here. 

\begin{proof}

Suppose the crossing $c$ is positive. Then the complex can be written 

\begin{figure}[h!]
\centering
\begin{tikzpicture}
\matrix(m)[matrix of math nodes,
row sep=10em, column sep=12em,
text height=1.5ex, text depth=0.25ex]
{C_{Lee}(D_{+}^{0})&C_{Lee}(D_{+}^{1})\\};
\path[-{stealth}]
(m-1-1) edge node[above] {$\delta_{+}$} (m-1-2);
\end{tikzpicture}
\end{figure}

\noindent
where $D_{+}^{0}$ and $D_{+}^{1}$ are the 0- and 1-resolution at $c$.

We define $H: C_{Lee}(D) \to C_{Lee}(D) $ to be equal to $0$ on $C_{Lee}(D^0_+)$ summand and $\delta_-$ on $C_{Lee}(D^1_+)$ summand, i.e. the edge map that would have appeared if $c$ were a negative crossing. Our complex with the total differential $\delta+H$ now looks like

\begin{figure}[h!]
\centering
\begin{tikzpicture}
\matrix(m)[matrix of math nodes,
row sep=10em, column sep=12em,
text height=1.5ex, text depth=0.25ex]
{C_{Lee}(D^{0}_+)&C_{Lee}(D^{1}_+)\\};
\path[-{stealth}]
(m-1-1) edge[bend left = 15] node[above] {$\delta_{+}$} (m-1-2)
(m-1-2) edge[bend left = 15] node[below] {$\delta_{-}$} (m-1-1);
\end{tikzpicture}

\end{figure}

\noindent
Since $\delta_{-}$ anticommutes with all edge maps except $\delta_{+}$, $\delta H+H \delta = \delta_{+}\delta_{-}+\delta_{-}\delta_{+}$, which by inspection is equal to $X_{j}+X_{k}$. 

The negative crossing argument is similar. In this case, the chain homotopy $H$ is defined using $\delta_+$ instead of $\delta_{-}$.

\end{proof}

Putting the previous two lemmas together, we get the following corollary.

\begin{cor} For any $1\le \mathsf{i}\le m$, any $a$ in $H_{Lee}(D_{+})$ and any $b$ in $H_{Lee}(D_{-})$, we have $(g_{*} \circ f_{*})(a)=\pm 2X_{\mathsf{i}}a$ and $(f_{*} \circ g_{*})(b)=\pm2X_{\mathsf{i}}b$, where the sign depends on $\mathsf{i}$.
\end{cor}

\section{Unknotting Number Bounds and the Knight Move Conjecture}

Suppose $C$ is a $\mathbb{Q}[X]$-module. Recall that the $X$-torsion in $C$, which we will denote by $T_{X}(C)$, is given by 

\[  T_{X}(C) = \{a \in C: X^{n}a=0 \text{ for some } n \in \mathbb{N}  \} \]

\noindent
We define $\mathfrak{u}_{X}(C)$ to be the maximum order of a torsion element in $C$. 

\begin{lem}

Let $D_{+}$ and $D_{-}$ be two knot diagrams which differ at a single crossing $c$. Then 

\[ | \mathfrak{u}_{X}(H_{Lee}(D_{+})) - \mathfrak{u}_{X}(H_{Lee}(D_{-}))| \le 1 \]

\noindent
where $X$ refers to $X_{i}$ for some $i$.
\end{lem}

\noindent
Note that up to sign, multiplication by any $X_{i}$ is the same on the Lee homology, so $\mathfrak{u}_{X_{i}}(H_{Lee})$ does not depend on the choice of $i$.

\begin{proof}

Let $a \in T_{X_{1}}(H_{Lee}(D_{+}))$, and let ord$_{X_{1}}(a)$ denote the order of $a$ with respect to $X_{1}$. Then 

\[ \text{ord}_{X_{1}}(a) \ge \text{ord}_{X_{1}}(f_{*}(a)) \ge \text{ord}_{X_{1}}(g_{*}(f_{*}(a)))    \]

\noindent
Since $g_{*}(f_{*}(a))=\pm 2X_{1}a$ and we're working over $\mathbb{Q}$, we get $\text{ord}_{X_{1}}(g_{*}(f_{*}(a)))=\text{max}(\text{ord}_{X_{1}}(a)-1, 0)$. This gives 

\[ \text{ord}_{X_{1}}(a)-1 \le \text{ord}_{X_{1}}(f_{*}(a)) \]

\noindent
so $ \mathfrak{u}_{X_{1}}(H_{Lee}(D_{+})) - \mathfrak{u}_{X_{1}}(H_{Lee}(D_{-})) \le 1$. The reverse inequality is obtained by starting with $b$ in $T_{X_{1}}(H_{Lee}(D_{-}))$ and applying $f_{*} \circ g_{*}$.

\end{proof}

\begin{thm}

For any knot $K$, $\mathfrak{u}_{X}(K)$ gives a lower bound for the unknotting number of $K$.

\end{thm}

\begin{proof}

This follows immediately from the previous lemma together with the observation that $H_{Lee}(Unknot) = \mathbb{Q}[t] \oplus \mathbb{Q}[t]$, so $\mathfrak{u}_{X}(Unknot)=0$.

\end{proof}

To translate this result back to the Lee spectral sequence, we use the fact that $X^{2}=t$. It follows that 

\[  \left \lceil{\mathfrak{u}_{X}(K)/2 }\right \rceil = \mathfrak{u}_{t}(K)\]

\noindent
where $\left \lceil{x }\right \rceil$ is the ceiling of $x$. The Lee spectral sequence collapses at the $E_{2}$ page if and only if $\mathfrak{u}_{t}(K)=1$.

\begin{thm}

If $K$ is a knot with $u(K) \le 2$ and $K$ is not the unknot, then the Lee spectral sequence for $K$ collapses at the $E_{2}$ page.

\end{thm}

\begin{proof}

By the previous theorem, $\mathfrak{u}_{X}(K) \le 2$, so $\mathfrak{u}_{t}(K) \le 1$. Since Khovanov homology detects the unknot, we know that $\mathfrak{u}_{t}(K) \ne 0$. The theorem follows.

\end{proof}

\begin{cor}

The Knight Move Conjecture is true for all knots $K$ with $u(K) \le 2$.

\end{cor}

\bibliographystyle{alpha}

\bibliography{sample}

\end{document}